\documentclass[12pt]{amsart}
\usepackage{amscd}
\usepackage{verbatim}
\usepackage{amssymb, amsmath, amsthm, amscd,ifthen}
\usepackage[dvips]{graphics}
\usepackage[cp866]{inputenc}
\usepackage{graphicx}
\usepackage{epsfig}
\usepackage{xcolor}
\usepackage{url}

\usepackage{amsmath,amssymb,amscd,}
\usepackage[all]{xy}

\usepackage{amssymb}

\newtheorem{thm}{Theorem}[section]

\newtheorem{prop}[thm]{Proposition}

\newtheorem{Ex}[thm]{Example}

\theoremstyle{definition}
\newtheorem{rem}[thm]{Remark}
\newtheorem{dfn}[thm]{Definition}

\title{Lower bounds on the number \\ of envy-free divisions}

\author[D. Joji\'{c}]{Du\v{s}ko Joji\'{c}}
\author[G. Panina]{Gaiane Panina}
\author[R. \v{Z}ivaljevi\'{c}]{Rade \v{Z}ivaljevi\'{c}}

\address[D. Joji\'{c}]{ Faculty of Science, University of Banja Luka}
\address[G. Panina]
{ St.\ Petersburg Department of Steklov Mathematical Institute}
\address[R. \v{Z}ivaljevi\'{c}]{Mathematical Institute of the Serbian Academy of Sciences and Arts (SASA), Belgrade}

\address{}

\subjclass{05E45, 52C99, 55M20, 91B32}
\keywords{Envy-free division, KKM theorem, chessboard complexes, configuration space/test map scheme}

\begin{document}

\date{April 26, 2025}

\begin{abstract}
We analyze lower bounds for the number of envy-free divisions, in the classical Woodall-Stormquist setting and in a non-classical case, when envy-freeness is combined with the equipartition of a measure.

 1. In the first scenario, there are $r$ hungry players, and the cake  (that is, the segment $[0,1]$) is cut into $r$ pieces. Then there exist at least two different envy-free divisions.  This bound is sharp: for each $r$, we present an example of preferences such that there are exactly  two envy-free divisions.

  2. In the second (hybrid) scenario, there are $p$ not necessarily hungry players ($p$ is a prime) and a continuous measure $\mu$ on $[0,1]$. The cake is cut into $2p-1$ pieces, the pieces are allocated to $p$ boxes (with some restrictions) and the players choose the boxes. Then there exists at least
  {$\binom{2p-1}{p-1} \cdot 2^{2-p}$ } envy-free divisions such that the measure $\mu$ is equidistributed among the players.

\end{abstract}

 \maketitle \setcounter{section}{0}

\section{Introduction}

Given some resource, identified with the unit segment $I=[0,1]$ (which is referred to as the \textit{cake}) and a set of players, one of the goals of
{welfare economics} is to divide the resource among them in an \emph{envy-free} manner.
Envy-freeness is the principle where every player feels that their share is at least as good as the share of any other player, and thus nobody feels envy.

\medskip

A \textit{cut} of the cake  is a sequence of numbers $x = (x_1,\dots, x_{r})$ where
\begin{equation}\label{eqn:cut}
    \sum _{i=1}^r x_i =1, \ \ x_i\geq 0,
\end{equation}
so the set of all cuts is naturally identified as the standard $(r-1)$-dimensional simplex $\Delta^{r-1}$.

The pieces of the cake arising from the cut (\ref{eqn:cut}) are the closed intervals (\textit{tiles})   whose lengths are $x_i$.

\medskip

 After a cut $x$ is chosen, each of the players expresses her individual \emph{preferences}, over the tiles arising from this particular cut, by pointing to one (or more) intervals which they like more than the rest.

  This idea is formalized in the following definition, where the individual preferences are interpreted as subsets of the simplex $\Delta^{r-1}$.

\begin{dfn}{\rm (\cite[Definition 2.2]{PaZi})}
The preferences of $r$ players is a matrix of subsets $\mathcal{A}=(A^j_i)_{i,j=1}^r$ of the standard simplex $\Delta^{r-1}$. The subsets are interpreted as preferences of the players in the following sense:
\begin{equation}\label{eqn:prefs-1}
x\in A^j_i \ \ \Leftrightarrow  \, \mbox{ {\rm in the cut} } x \mbox{ {\rm the player} } j \mbox{ {\rm prefers the tile} } i \, .
\end{equation}
\end{dfn}

\medskip
The classical KKM (Knaster-Kuratowski-Mazurkiewicz) assumptions on the preferences are the following:

\begin{itemize}
\item[$\mathbf{(P_{cl})}$] The preferences (the sets $A_i^j$) are \emph{closed}.
\item[$\mathbf{(P_{cov})}$] The preferences form a covering: $\bigcup_{i=1}^r A_i^j=\Delta^{r-1}, \ \ \forall j=1,...,r$. This means that whatever a cut is, each player  is expected to prefer at least one of the offered tiles.
    \item[$\mathbf{(P_{hung})}$] The players  are hungry: they never prefer degenerate tiles. That is,  $A_i^j$ does not intersect the facet of $\Delta^{r-1}$  defined by $x_i=0$.
\end{itemize}

\medskip
A \textit{division} of the interval $[0,1]$  is, by definition, a pair consisting of  a cut of the cake, and  an \textit{allocation}, that is, a bijection between  tiles and players.

\subsection{Envy-free division with either secretive or expelled players}


The history of these results (and their close relatives) can be summarized as follows.

\begin{enumerate}
\item[(a)]
D.R.\! Woodall proved (already in \cite{Wood}) a refined version of the envy-free-division-theorem with a \emph{secretive player}.
\item[(b)] Woodall's result was rediscovered by Asada et al.\ \cite{AsadaFrick17}, with a much simpler proof, under the name \emph{Strong colorful KKM theorem}.
\item[(c)] Meunier and Su, applying the method of multilabeled Sperner lemmas, offered \cite[Theorem 2.4]{MeSu} a unified point of view on envy-free division with secretive players, as well as the result where some of the players (unknown in advance) are ``voted off the group''.
\item[(d)] Closely related are the problems of ``envy-free division in the presence of a dragon'', studied in \cite{PaZi2022}, in a broader setting of non-hungry players.
\end{enumerate}

We shall need the following two theorems:
\begin{thm}\label{SecretiveThm}{\rm (Secretive player) \cite{Wood, AsadaFrick17, MeSu, PaZi2022}}
There are $r-1$ players and the cake is divided into $r$ pieces. If the preferences of players $\{A_i^j\}_{i\in [r]}, j\in [r-1]$,  satisfy classical KKM conditions,
then there exists a cut $x\in \Delta^{r-1}$ and $r$ bijections $\sigma_i : [r-1] \rightarrow [r]\setminus \{i\}, \, (i\in [r])$, such that for each $i\in [r]$,
\[
  x \in A^1_{\sigma_i(1)}\cap A^2_{\sigma_i(2)}\cap\dots\cap  A^{r-1}_{\sigma_i(r-1)} \, .
\]

In simple words, there exists a cut such that whatever tile is taken by the secretive player $r$, the rest of the tiles can be allocated in an envy-free manner.
\end{thm}

\begin{thm}\label{VotedOffThm}{\rm (Expelled player)}
There are $r+1$ players and the cake is divided into $r$ pieces. If the preferences of players $\{A_i^j\}_{i\in [r]}, j\in [r+1]$,  satisfy classical KKM conditions,
then there exists a cut $x\in \Delta^{r-1}$ and $r+1$ bijections $\tau_j : [r] \rightarrow [r+1]\setminus \{j\}, \, (j\in [r+1])$, such that for each $j\in [r+1]$,
\[
  x \in A_1^{\tau_j(1)}\cap A_2^{\tau_j(2)}\cap\dots\cap  A_r^{\tau_j(r)} \, .
\]

In simple words, there exists a cut such that, whoever is the expelled player, the   tiles can be allocated to the remaining players in an envy-free manner.
\end{thm}

\subsection{The central new results in the paper}

The problem of estimating the number of envy-free divisions doesn't seem to have received much attention in the existing literature.
Here we analyze two problems of this type, both fairly well-known and studied with the emphasis on the existence of a solution.

\medskip
By shifting the emphasis to the number of solutions here we determine:

\begin{enumerate}
\item A sharp lower bound for the number of envy-free division in the classical Woodall-Stormquist setting, i.e. for players with standard KKM preferences.

\item Lower bounds for combined envy-free division and equipartition of a measure, in the spirit of \cite{jpz-2}.
\end{enumerate}

\medskip\noindent
(1) An easy observation (Proposition \ref{thm:at-least-2}) states that each envy-free division problem in the Woodall-Stormquist setting has at least two distinct solutions. This
bound is sharp for each number of players, as shown by an example constructed in Theorem \ref{thm:not-more-than-2}.

In Section \ref{sec:extremal} we focus on preferences which are \emph{extremal} in the sense that they have precisely two distinct envy-free divisions.
We show (Proposition \ref{prop:exactly-one}) that extremal preferences have only one cut, allowing  two different envy-free allocations of pieces.

These two allocations define a bipartite graph on the sets of players and pieces. As shown in Theorem \ref{thm:extreme}, the only bipartite graph that arises in extremal preferences by this construction is a complete cycle (of size $2r$).

\bigskip\noindent
(2) In Section \ref{sec:simultaneous} we focus on a hybrid problem where the number of cuts (and tiles) is roughly twice as big as the number of players.
This allows us (Theorem \ref{thm:chessboard}) to achieve both an envy-free division among the players and an equipartition of a measure, which is prescribed in advance. As before, the emphasis is not only on the existence of a solution but includes an estimate of the number of simultaneous envy-free divisions and measure equipartitions.

It is interesting that in this case the lower bound  {$\binom{2p-1}{p-1} \cdot 2^{2-p}$} is much larger, illustrating different behaviour of relatively similar problems of envy-free division.

\section{The existence of two different envy-free divisions}

An easy corollary of the ``secretive player theorem'' (Theorem \ref{SecretiveThm}) says that, in the classical setting, there always exist at least two different envy-free divisions (Proposition \ref{thm:at-least-2}).
This result is in general the best possible, as shown by the example constructed in Theorem \ref{thm:not-more-than-2}.

\bigskip

\begin{prop}\label{thm:at-least-2}
  For the classical setting of hungry players, there are always at least two different envy-free divisions.
\end{prop}
\begin{proof}
  This follows directly from "secretive player" Theorem {\ref{SecretiveThm}}. Assume there is exactly one cut which gives an envy-free allocation, and prove that in this case this cut gives at least two envy-free allocations. Indeed, we may think that the first player is secretive. This means that each of the players $2,3,...,r$
  prefer at least two tiles that arise from this cut. Now let us think that player 2 is secretive. He may take any of the two tiles he prefers, and we
  have an allocation of the rest.
\end{proof}

\begin{Ex}\label{ex:exactly-two}{\rm  For $r=3$ there exist preferences that have exactly two different envy-free divisions, see Fig. \ref{fig:exactly-two}.}
\end{Ex}\label{ex:exactly-two}

\begin{figure}[htb]
\centering\vspace{+0cm}
\includegraphics[scale=0.35]{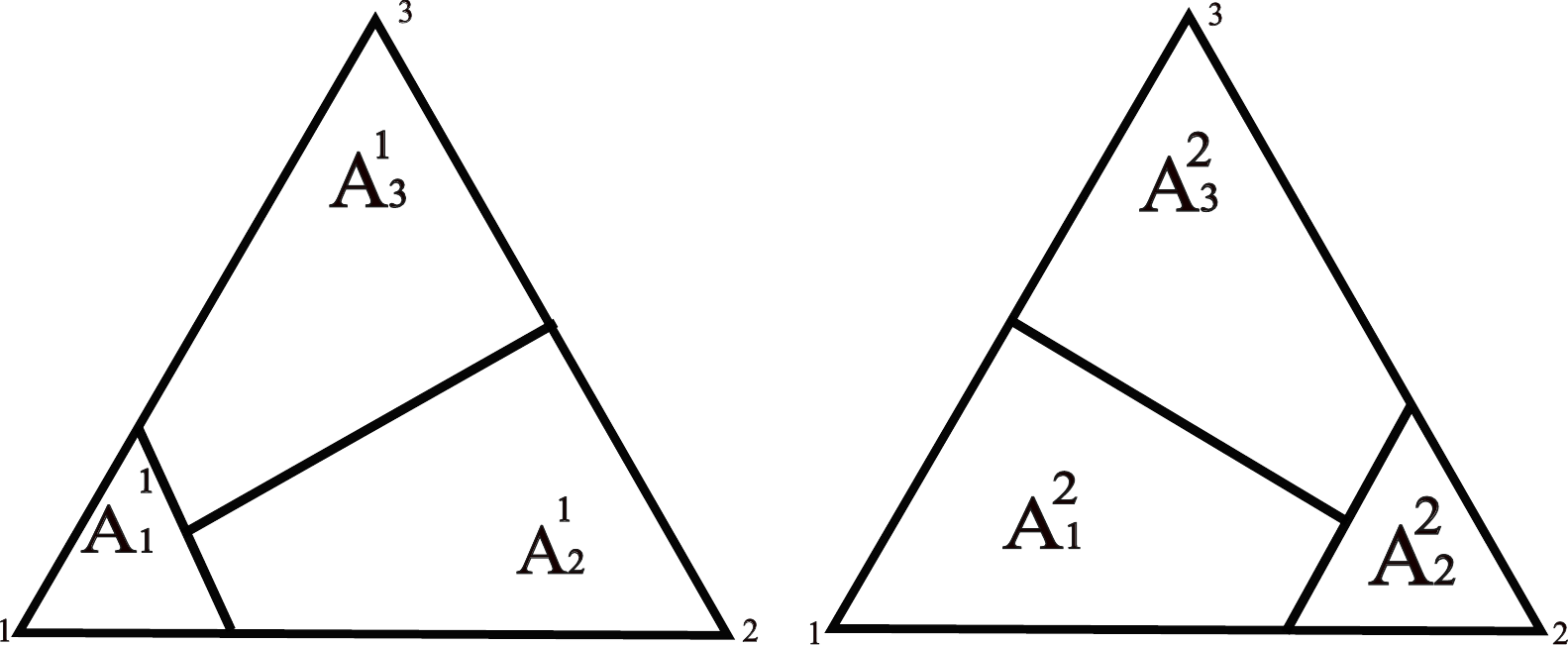}
\caption{{We depict the preferences of the first two players, the third ones are analogous.}}
\label{fig:exactly-two}
\end{figure}

Indeed, in the case of preferences exhibited in Figure \ref{fig:exactly-two}, the only common point (admissible cut) is the center of the triangle $O = (\frac{1}{3}, \frac{1}{3}, \frac{1}{3})$, while the corresponding envy-free divisions are:
\[
     A_1^1 \cap A_3^2 \cap A_2^3 = \{O\} =  A_3^1 \cap A_2^2 \cap A_1^3 \, .
\]

\medskip
The following theorem generalizes Example \ref{ex:exactly-two} to any number of players.

\begin{thm}\label{thm:not-more-than-2}
  For any $r$ there exist preferences that give exactly two different envy-free divisions.
\end{thm}
\begin{proof}
We construct a set (matrix) $(A_i^j)$ of preferences with exactly two envy-free divisions, which have one and the same cut of the cake $[0,1]$. In other words, the cut is the same but the tiles can be allocated in an envy-free fashion in two different ways.

 Choose a descending sequence ``small numbers'' $\{\varepsilon_i\}_{i=1}^{r-2}$, which satisfy the inequality
$$1/r>\varepsilon_1>\varepsilon_2>...>\varepsilon_{r-2}>0\, .$$

\medskip
First define the preferences  $A_i^1\subset \Delta^{r-1}$ of the first player:
\begin{align*}
& A_1^1   :=  \{x\in \Delta^{r-1} \mid x_1 \geqslant 1-\varepsilon_1\} \hbox{ which we shorten as } \{ x_1 \geqslant 1-\varepsilon_1\}\\
& A_2^1   :=  \{ x_1 \leqslant 1-\varepsilon_1,\, x_1 + x_2 \geqslant 1-\varepsilon_2 \} \\
& A_3^1   :=  \{ x_1 \leqslant 1-\varepsilon_1,\, x_1 + x_2 \leqslant 1-\varepsilon_2,\, x_1 + x_2 + x_3 \geqslant 1-\varepsilon_3 \} \\
&\dots \dots \dots \dots \dots \dots \dots \dots \dots \dots \dots \dots \dots \dots \dots \dots \dots \dots\\
& A_{r-2}^1   :=  \{ x_1 \leqslant 1-\varepsilon_1,\, \dots,\,  \sum_{i=1}^{r-3} x_i \leqslant 1-\varepsilon_{r-3},\, \sum_{i=1}^{r-2} x_i \geqslant 1-\varepsilon_{r-2} \} \\
&\hbox{At the last two steps we change the rule and set:}\\
& A_{r-1}^1   :=  \{ x\in \Delta^{r-1} \mid x_{r-1} \geqslant x_r \} \setminus Int \Big(\bigcup_{i=1}^{r-2} A_i^1\Big) \\
& A_{r}^1   :=  \{ x\in \Delta^{r-1} \mid x_{r} \geqslant x_{r-1} \} \setminus Int \Big(\bigcup_{i=1}^{r-2} A_i^1\Big)
\end{align*}

The preferences $A_i^j$ for $j\geq 2$ are defined by the substitution
\begin{equation}\label{eq:cyclic}
  A_i^1 \longrightarrow A_{i+j-1}^j, \quad x_i \rightarrow x_{i+j-1}
\end{equation}
or more explicitly,
\begin{align*}
& A_j^j   :=  \{ x_j \geqslant 1-\varepsilon_1\}\\
& A_{j+1}^j   :=  \{ x_j \leqslant 1-\varepsilon_1,\, x_j + x_{j+1} \geqslant 1-\varepsilon_2 \} \\
&\dots \dots \dots \dots \dots \dots \dots \dots \dots \dots \dots \dots \dots \dots \dots \dots \dots \dots\\
& A_{j+r-3}^j   :=  \{ x_j  \leqslant 1-\varepsilon_1,\, \dots,\,  \sum_{i=1}^{r-3} x_i \leqslant 1-\varepsilon_{r-3},\, \sum_{i=1}^{r-2} x_i \geqslant 1-\varepsilon_{r-2} \} \\
& A_{j+r-2}^j   :=  \{ x\in \Delta^{r-1} \mid x_{j+r-2} \geqslant x_{j+r-1} \} \setminus Int \Big(\bigcup_{i=1}^{r-2} A_{i+j-1}^j\Big) \\
& A_{j+r-1}^j   :=  \{ x\in \Delta^{r-1} \mid x_{j+r-1} \geqslant x_{j+r-2} \} \setminus Int \Big(\bigcup_{i=1}^{r-2} A_{i+j-1}^j\Big)
\end{align*}
The substitution (\ref{eq:cyclic}) means that there is an action of the cyclic group of order $r$ on the set of preferences. For this reason the indices in $A_i^j$ and $x_i$ are always understood modulo $r$.

\medskip
We begin the proof by showing that the only possible envy-free divisions are those where the corresponding  cut is the barycenter $b_0=(\frac 1 r,\frac 1 r,\ldots, \frac 1 r)$. Assuming the opposite, let $x=(x_1,x_2,\ldots,x_r)\in \Delta^{r-1}\setminus \{b_0\}$ be a different cut of the cake, which admits an envy-free division.

\medskip
Let $m = \max \{x_i\}$ and let $k_0$ be an index such that $x_{k_0-1}< m = x_{k_0}$. In light of the cyclic group action (\ref{eq:cyclic}), we can assume that $k_0=1$
and $x_{r}< m = x_{1}$. Also note that, by assumption, $m > \frac{1}{r} > \varepsilon_1$.

\medskip\noindent
\emph{Case 1.} This is a warming-up case which highlights the idea of the proof. Let us first  analyze the case $m=x_1 > x_2$.

\medskip From the viewpoint of the player $2$, none of the tiles $x_2, x_3, \dots, x_{r-1}$ is preferred. Indeed, this follows from the inequalities
\begin{align*}
  & x_2 < 1 - \varepsilon_1 \\
  & x_2 + x_3 < 1 - \varepsilon_2 \\
  & \dots \dots \\
  & x_2 + x_3 + \dots + x_{r-1} < 1 - \varepsilon_{r-2}
\end{align*}
which are an immediate consequence of $1 = x_1+x_2+\dots+ x_r$ and the assumption $m=x_1> \varepsilon_1$. It follows that the player 2 chooses one (or both) of the remaining tiles $x_1$ or $x_r$, as prescribed by preferences $A_r^2$ and $A_{r+1}^2 = A_1^2$. Finally, since $x_1> x_r$, the tile $x_1$ is the only preferred choice of player $2$.

\medskip
The situation with player $3$ is similar. In light of the inequalities
\begin{align*}
  & x_3 < 1 - \varepsilon_1 \\
  & x_3 + x_4 < 1 - \varepsilon_2 \\
  & \dots \dots \\
  & x_3 + x_4 + \dots + x_{r} < 1 - \varepsilon_{r-2}
\end{align*}
the player 3 may choose only { $x_2$ } or $x_1$, which means that $x_1$ is her only choice. But this is a contradiction, since we have two different players whose only preferred choice is the same tile $x_1$.

\medskip\noindent
\emph{Case 2.} Now turn to the general case, that is,  suppose that for some $k>1$
\[
x_r< x_1 = x_2 = \dots = x_k > x_{k+1} \, .
\]
Then by similar analysis as in \emph{Case 1}, we obtain that player 2 prefers only the tile $x_1$. The situation with player 3 is slightly different, she may now prefer either of the tiles $x_1$ or $x_2$. Similarly, the player $4$ prefers tiles $x_2$ and {  $x_3$} etc. This doesn't change until the player $k+1$ who may choose tiles $x_k$ and {  $x_{k-1}$}. Finally, the player $k+2$ may choose only the tile $x_k$.

The conclusion is that the players $\{2,3,\dots, k+1, k+2\}$ may only be matched with the tiles from the set $\{x_1, x_2, \dots , x_k\}$, which is a contradiction (violation of Hall's matching condition).
\end{proof}

\section{Extremal preferences}
\label{sec:extremal}

\begin{dfn}\label{dfn:extremal}
Preferences $\mathcal{A}=(A^j_i)_{i,j=1}^r$ are called extremal if they  admit exactly two envy-free divisions.
\end{dfn}

\begin{prop}\label{prop:exactly-one}
If $\mathcal{A}=(A^j_i)_{i,j=1}^r$ is an extremal collection of preferences then the following is satisfied:
\begin{enumerate}
\item There exists exactly one cut $x$ which allows an envy-free allocation. In other words the two envy-free divisions differ only in how the pieces of
the one and the same cut are allocated to the players.
\item For this unique cut, each player prefers at least two tiles.
\item For this unique cut, each tile is preferred by at least two players.
\end{enumerate}

\end{prop}

\begin{proof}
The case of two players is elementary, so we assume  $r\geq 3$.

(1, 2)  Assume there are two cuts, $x$ and $x'$. In this case each of them necessarily has a unique envy-free allocation.  Take the player $i$ and make him secretive.
We conclude that for each of the two cuts the player $i$ prefers exactly one tile. Otherwise one has two different allocations for one of the cuts.
 Now let us eliminate player $1$ and replace him by a secretive player. There exists a cut which guaranties an envy-free allocation for any choice of the secretive player.
This cut is necessarily equal to one of  $x$ and $x'$; let it be $x$. Now the secretive player acts arbitrarily, so  for the cut $x$ each of the players prefers at least two tiles. A contradiction.

(3) Assume the tile $i$ is preferred by a unique player $r$. Consider the following setting: there are $r+1$ players, the extra player $r+1$ has the same preferences as the player $r-1$. The cake is divided in $r$ pieces. Then one of the players is expelled. Theorem \ref{VotedOffThm} states that there exists a cut which guarantees an envy-free division.  Assume the player $r+1$ is expelled, therefore the cut is unique, and it is $x$.  If the player $r-1$ is expelled,  the tile $i$ remains unpreferred by any of the players.
 A contradiction.
\end{proof}

\medskip

\begin{dfn}
  Let $(A_i^j)$ be an extreme collection of preferences.  Take the unique cut that gives an envy-free division.  The graph $\Gamma=\Gamma(A_i^j)$ is a bipartite graph with $r$ black vertices (they correspond to players) and  $r$ white vertices (they correspond to tiles).
  A black and a white vertices share an  edge iff the corresponding player prefers the corresponding tile.
\end{dfn}

\begin{thm}\label{thm:extreme}
  For an extreme collection of preferences $(A_i^j)$, the graph $\Gamma(A_i^j)$ is a cycle.
\end{thm}

\begin{proof}
  Each envy-free division is a perfect matching on $\Gamma$. So there are
  exactly two perfect matchings, call them\textit{ red} and\textit{ blue}. After taking their union, one imagines that some of the edges are colored red, some are colored blue, some are red-blue, and some remain uncolored.

  Eliminating all uncolored edges yields a graph $\Gamma'$ which splits into connected components. There are two possible types of connected components:
  (1) cycles with alternating red and blue edges, and (2)  one-edge graphs  that correspond to red-blue edges.

  \medskip

 Two  observations are:
 \begin{enumerate}
   \item $\Gamma'$ has at least one cycle. Otherwise the red and blue perfect matchings coincide.

   \item  If  $\Gamma'$ has more than one cycle, $\Gamma'$  supports at least $4$ perfect matchings, since for two cycles one can combine ''red-red'',  ''red-blue'', ''blue-red'' and ''blue-blue''  matchings on thecycles.
 \end{enumerate}

     So $\Gamma'$ has exactly one cycle; we call it the \textit{alternating cycle}. Let us convince ourselves that $\Gamma'$ has no red-blue edges. Assume there is a  red-blue edge $e$. Since each of the vertices of $\Gamma$ has valency at least $2$, the edge $e$ participates in a path  $P$ in the graph $\Gamma$ such that in the path uncolored edges alternate with red-blue ones. There are two options:\begin{enumerate}
                                            \item The path $P$ closes to a cycle. Then the cycle admits yet another perfect matching. A contradiction.
                                            \item The path $P$ terminates by its two endpoints at the alternating cycle. Then it is  possible to create a new perfect matching  on their union, see Figure \ref{fig:red-and-blue}.
                                          \end{enumerate}
 It is an easy exercise  to prove that $\Gamma=\Gamma'$, that is, we have one alternating cycle and no more other (uncolored) edges.
\end{proof}

\begin{figure}[htb]
\centering\vspace{+0cm}
\includegraphics[scale=0.35]{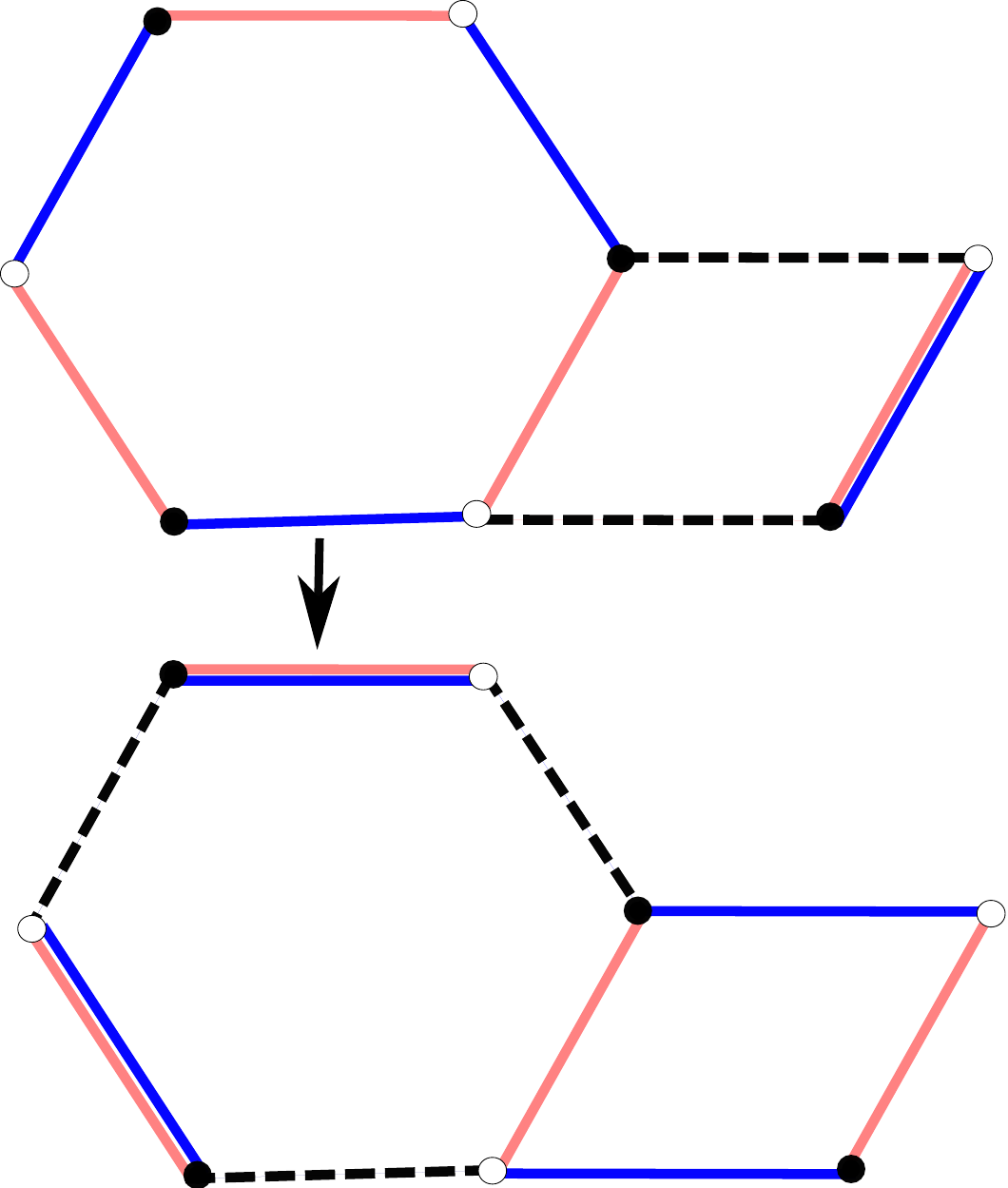}
\caption{{Generating new perfect matchings from old.}}
\label{fig:red-and-blue}
\end{figure}

\section{Simultaneous envy-free division and a measure equipartition}\label{sec:simultaneous}

In this section we assume that the number of players $p$ is a prime. Let $\mu$ be a continuous measure on  the segment (cake) $[0,1]$.
We aim not only at an envy-free division, but also at an equipartition of the measure. Clearly, $p-1$ cuts are insufficient, so
 we propose the following scenario:  the cake is cut into $2p-1$ tiles that are allocated to $p$ boxes, with some restrictions on the number of tiles in a box. The players express their preferences over the boxes (by inspection of the tiles inside the boxes).

\medskip
These are the so called \emph{new-style preferences}  $(B_i^j)$, as described in  \cite[Section 3.2]{PaZi}. The tiles are allocated to (labelled) boxes (by an allocation function $\alpha : [p] \rightarrow [p]$) and the preferences are expressed by the rule:
\begin{equation}\label{eqn:prefs-2}
\begin{split}
(x, \alpha) \in B^j_i \ \ \Leftrightarrow  \, &  \mbox{ {\rm in the cut} } x \mbox{ {\rm and the allocation} }
\alpha \mbox{ {\rm the player} } j\\
& \mbox{ {\rm prefers the content $\alpha^{-1}(i)$ of the box } } i \, .
\end{split}
\end{equation}

We refer the reader for a detailed exposition in \cite{PaZi}, briefly pinpointing important details:
\begin{enumerate}
\item The preferences $(B_i^j)$  are now subsets of  the configuration space $\mathcal{C}$ (described below).
\item In addition of being closed and covering, the preferences are also assumed to be \textit{equivariant}. In simple words, a player does not care about the numbering of the boxes, his preferences depend only on what is contained inside.
\item We do not assume that the players are hungry.
\item  It is assumed that degenerate tiles (if any) do not affect the preferences.
\end{enumerate}

\begin{dfn}
We say that a division in the above setting is  \textit{favourable} if
 \begin{enumerate}
 \item It is envy-free (every player gets a preferred box).
 \item The measure $\mu$ is equidistributed among the players
      \item Each player gets at most three tiles.
      \item The number of players who get three tiles is at most one.
    \end{enumerate}
\end{dfn}

\begin{rem}{\rm
 A favourable division of the cake may not exist if the cake is cut into $2p-2$ pieces.
As a consequence,  if the cake is cut into $2p-1$ pieces, and a favourable division exists, it never contains degenerate tiles.
}
\end{rem}
Here is a simple example (in the case $p=3$ and $2p-2=4$), illustrating this phenomenon. Suppose that
\[
  a_1< b_1 < a_2 < b_2 <a_3 < b_3 < c < d \, .
\]
The preference of player $P_i \, (i=1,2,3)$, is to maximize the measure of the interval $[a_i, b_i]$. The measure $\mu$, which is supposed to be equidistributed among the players, is the uniform measure on the interval $[c,d]$. This is possible only if the number of pieces is at least $2p-1=5$.

  \medskip

  \begin{thm}\label{thm:chessboard}  Let $p$ be a prime number and the preferences $B_i^j$ are closed and covering. Let $\mu$ be a continuous measure such that  there does not exist a favourable division of  $I=[0,1]$ with $2p-2$ tiles.
    Then there exist at least   {$\binom{2p-1}{p-1} \cdot{2^{2-p}}$} favourable divisions with $2p-1$ tiles.
  \end{thm}

\begin{proof} The proof follows the usual \emph{Configuration space/Test map}-scheme \cite{Z17}, modeled on examples described in \cite{PaZi}.

\medskip
Consider the following configuration space $\mathcal{C}$ (a generalized chessboard complex) on the $[p]\times [2p-1]$ chessboard. The vertices of
$\mathcal{C}$ are the elements of the chessboard $[p]\times [2p-1]$, while the simplices correspond to the rook placements which have at most one rook in each of $2p-1$ columns. In addition to that, in each row of the chessboard a simplex may have at most one rook in columns $1,2,\dots, p-1$ and at most one rook in columns $p,p+1,\dots, 2p-2$. (As a consequence in each row there are at most three rooks.)

\medskip
Informally speaking, the first $p-1$ columns of the chessboard are colored in red, the second $p-1$ columns are blue, and the last column is white. So the original chessboard
$[p]\times [2p-1]$ decomposes into three chessboards, positioned side by side, and for each of them we obtain the corresponding chessboard complex.

\medskip
From this description it easily follows that $\mathcal{C}$ is a join of standard chessboard complexes \cite{PaZi}
\begin{equation}\label{eq:coloring}
           \mathcal{C} \cong \Delta_{p,p-1} \ast \Delta_{p,p-1} \ast \Delta_{p,1}
\end{equation}
where the vertices of  $\Delta_{m,n}$ are elements of $[m]\times [n]$, and simplices correspond to non-attacking arrangements of rooks in the chessboard $[m]\times [n]$.

{The columns of the chessboard correspond to tiles of the cake, so we imagine that given a cut, the first $p-1$  tiles are red, the second $p-1$ tiles are blue, and the last tile is white. }

\medskip

The associated test map
\begin{equation}\label{eq:test}
F = (f_1, f_2) : \Delta_{p,p-1} \ast \Delta_{p,p-1} \ast \Delta_{p,1} \rightarrow W_1 \oplus W_2
\end{equation}
has two components,
\[
   f_1 : \mathcal{C} \rightarrow W_1 \mbox{ {and} } f_2 : \mathcal{C} \rightarrow W_2
\]
where $W_1 \cong W_2 \cong \mathbb{R}^p/ D$ is the $(p-1)$-dimensional $\mathbb{Z}_p$-representation, obtained by subtracting the diagonal $D$ from  $\mathbb{R}^p$.
The component $f_1$ accounts for the envy-free division, while $f_2$ tests for an equipartition of $\mu$\footnote{{The functions $f_1$ and $f_2$ arise in a routine way, in \textit{envy-free and fair necklace splitting theorems}, see  \cite{jpz-2}. In short, $f_1$ records the values of the measure $\mu$ of the boxes, and $f_2$ comes from partition of unity related to the prefrences  }.}

\medskip
The test map (\ref{eq:test}) must have a zero, as a consequence of the fact (see  \cite{vz11}) that the degree of the following ($\mathbb{Z}_p$-equivariant) map
\[
\Delta_{p,p-1} \ast \Delta_{p,p-1} \rightarrow S(W_1 \oplus W_2) \cong S(W_1) \ast S(W_2)
\]
is non-zero.   {By construction, each zero of the test map corresponds to a favourable division with additional restrictions:}

{(7) \ \ \ \ \ \ \textit{each player gets a box with all tiles of different colors.}}

   \bigskip
   
    {Secretive player theorem holds true: there exists a cut and an allocation to boxes, such that whatever box is taken by the secretive player $p$, the rest of the boxes can be allocated to the other players, and the division is favourable  and satisfies (7). The proof of this fact is routine; it uses either Birkhoff polytope or Hall marriage lemma. We refer the reader to \cite{PaZi2022} for details of a proof for a different configuration space. Therefore there exist at least two favourable divisions with  additional restriction (7).}

 \bigskip
    Now the idea is  to look at a similar (and isomorphic)  configuration space, obtained by a different choice of $(p-1)$ red  and $(p-1)$ blue columns in the chessboard $[p]\times [2p-1]$. Red (blue) columns correspond to the first (second) copy of $\Delta_{p,p-1}$ in (\ref{eq:coloring}).  Each way of coloring gives  { at least two} favourable division.


   There are $$\binom{2p-1}{p-1}\cdot p$$
   ways to choose $p-1$  red columns  and one white.
   The remaining $p-1$ columns are blue.

   Note that each coloring produces at least two favourable division. However, each favourable division arises with different colorings.

   There are two cases:\begin{enumerate}
                         \item A favourable division has no box with three tiles. Therefore each box (except for one box) contains  two tiles, and a unique box contains one tile. The number of colorings is
                             $$2^{p-1}+(p-1)2^{p-1}=p2^{p-1}.$$
                         \item There is a box with three tiles, $p-3$ boxes  with two tiles, and $2$ boxes with one tile.
                         The number of colorings is $$3\cdot2^{p-1}.$$
                       \end{enumerate}

   We conclude that for $p\geq 3$ the number of favourable divisions  is at least  {$$\frac{2\binom{2p-1}{p-1}\cdot p}{p\cdot 2^{p-1}}=\frac{\binom{2p-1}{p-1} }{2^{p-2}}.$$}

  \end{proof}

\subsection*{Acknowledgements} The authors are very grateful to CIRM, Luminy, where this research was initiated in
the framework of a ``Research in Pairs'' program, in March 2025. Section 4  is supported by the Russian Science Foundation (project 25-11-00058).

\end{document}